\documentclass[a4paper,11pt]{article}

\usepackage{amsmath}
\usepackage{amsfonts}
\usepackage{amssymb}
\usepackage{amsthm}

\newcommand{\p}{\varphi}
\newcommand{\R}{\mathbb{R}}
\newcommand{\weight}[1]{\langle #1\rangle}

\newcommand{\ZuWeis}{\mathrel{\mathop:\!\!=}}
\newcommand{\bv}{\mathbf{v}}
\newcommand{\tbj}{\widetilde{\mathbf{J}}}
\newcommand{\bpsi}{\boldsymbol{\psi}}
\newcommand{\bJ}{\mathbf{J}}
\newcommand{\N}{\ensuremath{\mathbb{N}}}
\newcommand{\Hone}{H^1_{(0)}}
\newcommand{\supp}{\operatorname{supp}}
\newcommand{\di}{\operatorname{div}}
\newcommand{\ol}[1]{\overline{#1}}
\newcommand{\sd}{\, d}
\newcommand{\e}{\varepsilon}

\newtheorem{definition}{Definition}[section]
\newtheorem{remark}[definition]{Remark}
\newtheorem{lemma}[definition]{Lemma}
\newtheorem{theorem}[definition]{Theorem}
\newtheorem{assumption}[definition]{Assumption}

\numberwithin{equation}{section}  %equation-counter is set back to zero at the beginning of each section

\begin{document}

\begin{titlepage}
\title{On an Incompressible Navier-Stokes/Cahn-Hilliard System with Degenerate Mobility}
\author{  Helmut Abels\footnote{Fakult\"at f\"ur Mathematik,  
Universit\"at Regensburg,
93040 Regensburg,
Germany, e-mail: {\sf helmut.abels@mathematik.uni-regensburg.de}}, Daniel Depner\footnote{Fakult\"at f\"ur Mathematik,  
Universit\"at Regensburg,
93040 Regensburg,
Germany, e-mail: {\sf daniel.depner@mathematik.uni-regensburg.de}}, and
Harald Garcke\footnote{Fakult\"at f\"ur Mathematik,  
Universit\"at Regensburg,
93040 Regensburg,
Germany, e-mail: {\sf harald.garcke@mathematik.uni-regensburg.de}}}
\date{}
\end{titlepage}
\maketitle
\begin{abstract}
 We prove existence of weak solutions for a diffuse interface model
 for the flow of two viscous incompressible Newtonian fluids in a bounded domain by allowing for a degenerate 
 mobility. The model has been developed by Abels, Garcke and Gr\"un for fluids with different densities and leads to 
 a solenoidal velocity field. It is given by a non-homogeneous Navier-Stokes system with a modified convective 
 term coupled to a Cahn-Hilliard system, such that an energy estimate is fulfilled which follows from the fact 
 that the model is thermodynamically consistent.
\end{abstract}
\noindent{\bf Key words:} Two-phase flow, Navier-Stokes equations,
 diffuse interface model, mixtures of viscous fluids, Cahn-Hilliard equation, degenerate mobility

\noindent{\bf AMS-Classification:} 
Primary: 76T99; %% Two-Phase flows: Others
Secondary: 35Q30, %% Stokes and Navier-Stokes eq.
35Q35, %% Other equations arising in fluid mechanics
76D03, %% Incompressible viscous fluids: Existence, uniqueness, and regularity theory 
76D05, %% Incompressible viscous fluids: Navier-Stokes equations
76D27, %% Incompressible viscous fluids: Other free-boundary flows; Hele-Shaw flows
76D45 %% Incompressible viscous fluids: Capillarity (surface tension)

%%%%%%%%%%%%%%%%%%%%%%%%%%%%%%%%%%%%%%%%%%%%%%%%%%%%%%%%%%%%%%%%%%%%%%%%%%%%%%%%%%%%%%%%%%%%%%%%
%%%%%%%%%%%%%%%%%%%%%%%%%%%%%%%%%%%%%%%%%%%%%%%%%%%%%%%%%%%%%%%%%%%%%%%%%%%%%%%%%%%%%%%%%%%%%%%%
%%%%%%%%%%%%%%%%%%%%%%%%%%%%%%%%%%%%%%%%%%%%%%%%%%%%%%%%%%%%%%%%%%%%%%%%%%%%%%%%%%%%%%%%%%%%%%%%
\section{Introduction} \label{intro}

Classically the interface between two immiscible, viscous fluids has
been modelled in the context of sharp interface approaches, see
e.g. \cite{Mue85}. But in the context of sharp interface models it is
difficult to describe topological changes, as e.g. pinch off and
situations where different interfaces or different parts of an
interface connect. In the last 20 years phase field approaches have
been a promising new approach to model interfacial evolution in
situations where interfacial energy effects are important, see
e.g. \cite{Che02}. In phase field approaches a phase field or order
parameter is introduced which rapidly changes its value in the
interfacial region and attains two prescribed values away from the
interface.

For two-phase flow of immiscible, viscous fluids a phase-field
approach first has been introduced by Hohenberg and Halperin
\cite{HH77}, the so-called ``Model H''. 
In their work the Cahn-Hilliard equation was coupled to
the Navier-Stokes system in such a way that capillary forces on the
interface are modelled with the help of the phase field. The approach
of Hohenberg and Halperin \cite{HH77} was restricted to the case where
the densities of the two fluids are the same or at least are very
close (``matched densities''). It has been later shown by Gurtin,
Polignone, Vi\~{n}als \cite{GPV96} that the model can be derived in the
context of rational thermodynamics. In particular global and local
energy inequalities are true. These global energy estimates can be
used to derive a priori estimates and this has been used by Boyer~\cite{Boy99}
and by Abels~\cite{Abe09b} for proofs of existence results. 

Often the densities in two phase flow are quite different. Therefore,
there have been several attempts to derive phase field models for two
phase flow with non-matched densities. Lowengrub and Truskinovsky
\cite{LT98} derived a first thermodynamically consistent phase field
model for the case of different densities. The model of Lowengrub and
Truskinovsky is based on a barycentric velocity and hence the overall
velocity field turns out to be not divergence free in general. In addition, the
pressure enters the Cahn-Hilliard equation and as a result the
coupling between the Cahn-Hilliard equation and the Navier-Stokes
equations is quite strong. This and the fact that the velocity field is
not divergence free makes numerical and analytical approaches quite
difficult. To the authors knowledge there have been so far no
numerical simulations for the full Lowengrub-Truskinovsky model. With
respect to analytical results we refer to the works of Abels~\cite{Abe09a, Abe12} for existence results. 

In a paper by Ding, Spelt and Shu \cite{DSS07} a generalization of
Model H for non-matched densities and a divergence free velocity
field has been derived. However it is not known whether this model is
thermodynamically consistent. A first phase field model for
non-matched densities and a divergence free velocity field which in
addition fulfills local and hence global free energy inequalities has
been derived by Abels, Garcke and Gr\"un \cite{AGG12}. The model in
\cite{AGG12} is given by the following system of
Navier-Stokes/Cahn-Hilliard equations:
\begin{align*}
 \partial_t (\rho(\varphi) \mathbf{v}) + \operatorname{div} ( \bv \otimes(\rho(\varphi) \bv + \tbj)) 
  &- \operatorname{div} (2 \eta(\varphi) D \bv) + \nabla p  & \\
  & = - \operatorname{div}(a(\varphi) \nabla \varphi \otimes \nabla \varphi)& \mbox{in } \, Q_T ,  \\
 \operatorname{div} \, \bv &= 0& \mbox{in } \, Q_T , \\
 \partial_t \varphi + \bv \cdot \nabla \varphi 
  &= \mbox{div}\left(m(\varphi) \nabla \mu \right)& \mbox{in } \, Q_T ,\\
 \mu = \Psi'(\varphi) + a'(\varphi) \frac{|\nabla \varphi|^2}{2} 
  &- \operatorname{div}\left( a(\varphi) \nabla \varphi \right)& \mbox{in } \, Q_T , 
\end{align*}
where $\tbj = -\frac{\tilde{\rho}_2 - \tilde{\rho}_1}{2} m(\varphi) \nabla \mu$, 
$Q_T=\Omega\times(0,T)$ for $0<T<\infty$, and $\Omega \subset \mathbb{R}^d$, $d=2,3$, is a sufficiently 
smooth bounded domain. We close the system  with the boundary and initial conditions
\begin{alignat*}{2}
 \bv|_{\partial \Omega} &= 0 &\qquad& \text{on}\ \partial\Omega\times (0,T), \\
 \partial_n \varphi|_{\partial \Omega} = \partial_n \mu|_{\partial \Omega} &= 0&& \text{on}\ \partial\Omega\times (0,T),  \\
 \left(\bv , \varphi \right)|_{t=0} &= \left( \bv_0 , \varphi_0 \right) &&\text{in}\ \Omega, 
\end{alignat*}
where $\partial_n \varphi = n\cdot \nabla \varphi$ and $n$ denotes the exterior normal at $\partial\Omega$.
Here $\mathbf{v}$ is the volume averaged velocity, 
$\rho=\rho(\varphi)$ is the
density of the mixture of the two fluids, $\varphi$ is the difference of the volume fractions of the two fluids
and we assume a constitutive relation between $\rho$ and the order parameter $\varphi$ given by
$\rho(\varphi) = \frac{1}{2} (\tilde{\rho}_1 + \tilde{\rho}_2) + \frac{1}{2} (\tilde{\rho}_2 - \tilde{\rho}_1) \varphi$, 
see \cite{ADG12} for details. In addition,
$p$ is the pressure, $\mu$ is the chemical potential associated to $\varphi$ and $\tilde{\rho}_1$, $\tilde{\rho}_2$ 
are the specific constant mass densities of the unmixed fluids.
Moreover, ${D}\mathbf{v}= \frac12(\nabla \mathbf{v} + \nabla \mathbf{v}^T)$,
$\eta(\varphi)>0$ is a viscosity coefficient, and $m(\varphi) \geq 0$ is a degenerate mobility coefficient. Furthermore,
$\Psi(\varphi)$ is the homogeneous free energy density for the mixture and the (total) free energy of the system is given by
\begin{align*}
E_{\mbox{\footnotesize free}}(\varphi) =   \int_{\Omega} \left( \Psi(\varphi) + a(\varphi)\frac{|\nabla \varphi|^2}{2} \right)\sd x
\end{align*}
for some positive coefficient $a(\varphi)$.
The kinetic energy is given by $E_{\mbox{\footnotesize kin}}(\p,\bv) \hspace*{-1pt}= \int_\Omega \rho(\varphi) \frac{|\bv|^2}{2} \, dx$ and the 
total energy as the sum of the kinetic and free energy is
\begin{align} 
\begin{split} \label{totalenergy}
 E_{\mbox{\footnotesize tot}}(\p,\bv)  
         &= E_{\mbox{\footnotesize kin}}(\p,\bv) + E_{\mbox{\footnotesize free}}(\p) \\
         &= \int_\Omega \rho(\varphi)\frac{|\bv|^2}{2} \, dx + \int_\Omega \left( \Psi(\p) + a(\varphi)\frac{|\nabla \varphi|^2}{2} \right) dx .
\end{split}
\end{align}
In addition there have been further modelling attempts for two phase
flow with different densities. We refer to Boyer \cite{Boy02} and the
recent work of Aki et al. \cite{ADGK12}. We remark that for the model
of Boyer no energy inequalities are known and the model of Aki et 
al.\,does not lead to velocity fields which are divergence free. 

In \cite{ADG12} an existence result for the above
Navier-Stokes/Cahn-Hil\-liard model has been shown in the case of a
non-degenerate mobility $m(\varphi)$. As is discussed in \cite{AGG12}
the case with non-degenerate mobility can lead to Ostwald ripening
effects, i.e., in particular larger drops can grow to the expense of
smaller ones. In many applications this is not reasonable and as
pointed out in \cite{AGG12} degenerate mobilities avoid Ostwald ripening
and hence the case of degenerate mobilities is very important in
applications. In what follows we assume that $m(\varphi) =
1-\varphi^2$ for $|\varphi| \leq 1$ and extend this by zero to all of
$\R$.  In this way we do not allow for diffusion through the bulk, i.e., the
region where $\varphi = 1$ resp. $\varphi = -1$, but only in the 
interfacial region, where $|\varphi|<1$. The degenerate mobility leads to the
physically reasonable bound $|\varphi| \leq 1$ for the order parameter
$\varphi$, which is the difference of volume fractions and therefore
we can consider in this work a smooth homogeneous free energy density
$\Psi$ in contrast to the previous work~\cite{ADG12}.

For the Cahn-Hilliard equations without the coupling to the
Navier-Stokes equations Elliott and Garcke~\cite{EG96} considered the
case of a degenerate mobility, see also Gr\"un \cite{Gru95}. We will
use a suitable testing procedure from the work \cite{EG96} to get a
bound for the second derivatives of a function of $\varphi$ in the
energy estimates of Lemma \ref{lem:energyestimate}. We point out that
our result is also new for the case of model $H$ with degenerate
mobility, i.e., $\tilde{\rho}_1 = \tilde{\rho}_2$, which implies $\widetilde{\bJ} =0$ in the above
Navier-Stokes/Cahn-Hilliard system.

The structure of the article is as follows: In Section 2 we summarize
some notation and preliminary results.  Then, in Section 3, we
reformulate the Navier-Stokes/Cahn-Hilliard system suitably, define weak solutions and state our
main result on existence of weak solutions. For the proof of the
existence theorem in Subsections 3.2 and 3.3 we approximate the
equations by a problem with positive mobility $m_\e$ and singular
homogeneous free energy density $\Psi_\e$. For the solution
$(\bv_\e,\varphi_\e,\bJ_\e)$ of the approximation (with $\bJ_\e = -
m_\e(\varphi_\e) \nabla \mu_\e$) we derive suitable energy estimates
to get weak limits. Then we extend the weak convergences to strong
ones by using methods similar to the previous work of the authors
\cite{ADG12}, careful estimates of the additional singular free energy density 
and by an additional subtle argument with the help of
time differences and a theorem of Simon \cite{Sim87}. We remark that
this last point would be easier in the case of a constant coefficient
$a(\varphi)$ in the free energy. Finally we can pass to the limit $\e
\to 0$ in the equations for the weak solutions
$(\bv_\e,\varphi_\e,\bJ_\e)$ and recover the identities for the weak
solution of the main problem.

%%%%%%%%%%%%%%%%%%%%%%%%%%%%%%%%%%%%%%%%%%%%%%%%%%%%%%%%%%%%%%%%%%%%%%%%%%%%%%%%%%%%%%%%%%%%%%%%
%%%%%%%%%%%%%%%%%%%%%%%%%%%%%%%%%%%%%%%%%%%%%%%%%%%%%%%%%%%%%%%%%%%%%%%%%%%%%%%%%%%%%%%%%%%%%%%%
%%%%%%%%%%%%%%%%%%%%%%%%%%%%%%%%%%%%%%%%%%%%%%%%%%%%%%%%%%%%%%%%%%%%%%%%%%%%%%%%%%%%%%%%%%%%%%%%
\section{Preliminaries and Notation} \label{prelimi}

We denote 
$a\otimes b = (a_i b_j)_{i,j=1}^d$ for $a,b\in \R^d$ and $A_{\operatorname{sym}}= \frac12 (A+A^T)$ for a matrix $A\in \R^{d\times d}$.
If $X$ is a Banach space and $X'$ is its dual, then 
\begin{equation*}
  \weight{f,g} \equiv \weight{f,g}_{X',X} = f(g), \qquad f\in X', \;g\in X,
\end{equation*}
denotes the duality product. We write $X\hookrightarrow \hookrightarrow Y$ if $X$ is compactly embedded into $Y$. 
Moreover, if $H$ is a Hilbert space, $(\cdot\,,\cdot )_H$ denotes its inner product. 
Moreover, we use the abbreviation $(.\,,.)_{M}=(.\,,.)_{L^2(M)}$.   

\medskip

\noindent {\bf Function spaces:}
If $M\subseteq \R^d$ is measurable, 
$L^q(M)$, $1\leq q \leq \infty$, denotes the usual Lebesgue-space and $\|.\|_q$
its norm.
Moreover, $L^q(M;X)$ denotes the set of all strongly measurable
$q$-integrable functions if $q \in [1,\infty)$ and essentially bounded strongly measurable functions, 
if $q = \infty$, where $X$ is a Banach space. 

Recall that, if $X$ is a Banach space with the Radon-Nikodym property, then 
\begin{equation*}%\label{eq:DualLq}
L^q(M;X)'= L^{q'}(M;X')\qquad \text{for every}\ 1\leq q < \infty 
\end{equation*}
by means of the duality product
$
  \weight{f,g}= \int_{M} \weight{f(x),g(x)}_{X',X}  dx 
$
for $f\in L^{q'}(M;X')$, $g\in L^{q}(M;X)$. If $X$ is reflexive or $X'$ is
separable, then $X$ has the Radon-Nikodym property, cf. Diestel and Uhl~\cite{DU77}. 

Moreover, we recall the Lemma of Aubin-Lions: If $X_0\hookrightarrow\hookrightarrow X_1 \hookrightarrow X_2$ 
are Banach spaces, $1<p<\infty$, $1\leq q <\infty$, and $I\subset \R$ is a bounded interval, then
\begin{equation}\label{eq:AubinLions}
  \left\{v\in L^p( I; X_0): \frac{dv}{dt} \in L^q(I;X_2) \right\}
  \hookrightarrow\hookrightarrow L^p(I;X_1).
\end{equation}
See J.-L.~Lions~\cite{Lio69} for the case $q>1$ and Simon~\cite{Sim87} or Roub{\'{\i}}\-{\v{c}}ek~\cite{Rou90} for $q=1$.

Let $\Omega \subset \R^d$ be a domain. 
Then $W^k_q(\Omega)$, $k\in \N_0$, $1\leq q\leq \infty$, denotes the usual $L^q$-Sobolev space, % $W^m_{q,\loc}(\ol{\Omega})$ its local version, 
$W^k_{q,0}(\Omega)$  the closure of $C^\infty_0(\Omega)$ in $W^k_q(\Omega)$,
$W^{-k}_q(\Omega)= (W^k_{q',0}(\Omega))'$, and $W^{-k}_{q,0}(\Omega)= (W^k_{q'}(\Omega))'$. 
We also use the abbreviation $H^k(\Omega) = W^k_2(\Omega)$.

Given $f\in L^1(\Omega)$, we denote by 
$
  f_\Omega = \frac1{|\Omega|}\int_\Omega f(x) \,dx
$
its mean value. Moreover, for $m\in\R$ we set
\begin{equation*}
  L^q_{(m)}(\Omega):=\{f\in L^q(\Omega):f_\Omega=m\}, \qquad 1\leq q\leq \infty.  
\end{equation*}
Then for $f \in L^2(\Omega)$ we observe that 
\begin{align*}
 P_0 f:= f-f_\Omega= f-\frac1{|\Omega|}\int_\Omega f(x) \,dx
\end{align*}
is the orthogonal projection onto $L^2_{(0)}(\Omega)$. 
Furthermore,
we define
\begin{equation*}
 \Hone\equiv\Hone (\Omega)= H^1(\Omega)\cap L^2_{(0)}(\Omega), \qquad (c,d)_{\Hone(\Omega)} := (\nabla c,\nabla d)_{L^2(\Omega)}.  
\end{equation*}
Then $\Hone(\Omega)$ is a Hilbert space due to Poincar\'e's inequality.

\medskip

\noindent
{\bf Spaces of solenoidal vector-fields:}
For a bounded domain $\Omega \subset \R^d$ we denote by $C^\infty_{0,\sigma}(\Omega)$ 
in the following  the space of all
divergence free vector fields in $C^\infty_0(\Omega)^d$ and
$L^2_\sigma(\Omega)$ is its closure in the $L^2$-norm. The corresponding
Helmholtz projection is denoted by $P_\sigma$,
cf. e.g. Sohr \cite{Soh01}. We note that $P_\sigma f = f- \nabla p$, 
where $p \in W^1_2(\Omega)\cap L^2_{(0)}(\Omega)$ is the solution of the weak Neumann problem 
\begin{equation}\label{eq:WeakHelmholtz}
  (\nabla p,\nabla \varphi)_{\Omega} = (f, \nabla \varphi)\quad \text{for all}\ \varphi \in C^\infty(\ol{\Omega}).
\end{equation}

\medskip

\noindent
{\bf Spaces of continuous vector-fields:}
In the following let $I=[0,T]$ with $0<T< \infty$ or let $I=[0,\infty)$ if $T=\infty$ and let $X$ be a Banach
space. Then $BC(I;X)$ is the Banach space of all bounded and continuous
$f\colon I\to X$ equipped with the supremum norm and $BUC(I;X)$ is the
subspace of all bounded and uniformly continuous functions. Moreover, we
define $BC_w(I;X)$ as the topological vector space of all bounded and weakly
continuous functions $f\colon I\to X$. By $C^\infty_0(0,T;X)$ we denote
the vector space of all smooth functions $f\colon (0,T)\to X$ with $\supp
f\subset\subset (0,T)$.
We say that $f\in W^1_p(0,T;X)$ for $1\leq p <\infty$, if and only if $f,
\frac{df}{dt}\in L^p(0,T;X)$, where $\frac{df}{dt}$ denotes the vector-valued
distributional derivative of $f$.
Finally, we note:
\begin{lemma}\label{lem:CwEmbedding}
  Let $X,Y$ be two Banach spaces such that $Y\hookrightarrow X$ and $X'\hookrightarrow Y'$ densely.
  Then $L^\infty(I;Y)\cap BUC(I;X) \hookrightarrow BC_w(I;Y)$.
\end{lemma}
\noindent
For a proof, see e.g. Abels \cite{Abe09a}.

%%%%%%%%%%%%%%%%%%%%%%%%%%%%%%%%%%%%%%%%%%%%%%%%%%%%%%%%%%%%%%%%%%%%%%%%%%%%%%%%%%%%%%%%%%%%%%%%%%%%%%%%%%
%%%%%%%%%%%%%%%%%%%%%%%%%%%%%%%%%%%%%%%%%%%%%%%%%%%%%%%%%%%%%%%%%%%%%%%%%%%%%%%%%%%%%%%%%%%%%%%%%%%%%%%%%%
%%%%%%%%%%%%%%%%%%%%%%%%%%%%%%%%%%%%%%%%%%%%%%%%%%%%%%%%%%%%%%%%%%%%%%%%%%%%%%%%%%%%%%%%%%%%%%%%%%%%%%%%%%
\section{Existence of Weak Solutions}

In this section we prove an existence result for the Navier-Stokes/Cahn-Hilliard system from the 
introduction for a situation with degenerate mobility. Since in this case we will not have a control
of the gradient of the chemical potential, we reformulate the equations by introducing a flux 
$\bJ = -m(\varphi) \nabla \mu$ consisting of the product of the mobility and the gradient of the
chemical potential. In this way, the complete system is given by:
\begin{subequations}
\begin{align}
 \partial_t (\rho \mathbf{v}) + \operatorname{div} (&\rho \bv \otimes \bv) - \operatorname{div} (2 \eta(\varphi) D \bv)
   + \nabla p  \nonumber \\
  &  + \operatorname{div}(\bv \otimes \beta \bJ) 
    =  - \operatorname{div}(a(\varphi) \nabla \varphi \otimes \nabla \varphi) & \mbox{in } \, Q_T , \label{degequ1} \\
 \operatorname{div} \, \bv &= 0 & \mbox{in } \, Q_T , \label{degequ2} \\
 \partial_t \varphi + \bv \cdot \nabla \varphi &= -\operatorname{div} \bJ & \mbox{in } \, Q_T , \label{degequ3} \\
 \bJ &= -m(\varphi) \nabla \left( 
      \Psi'(\varphi) + a'(\varphi) \frac{|\nabla \varphi|^2}{2} \right.\nonumber \\
   & \hspace*{60pt}  - \operatorname{div}\left( a(\varphi) \nabla \varphi \right)
                           \bigg) & \mbox{in } \, Q_T , \label{degequ4} \\
 %\mu &= \Psi'(\varphi) + a'(\varphi) \frac{|\nabla \varphi|^2}{2} - \operatorname{div}\left( a(\varphi) \nabla \varphi \right) & \mbox{in } \, Q , \label{degequ5} \\
 \bv|_{\partial \Omega} &= 0 & \mbox{on } \, S_T , \label{degequ5} \\
 \partial_n \varphi|_{\partial \Omega} &= (\bJ \cdot n)|_{\partial \Omega} = 0 & \mbox{on } \, S_T , \label{degequ6} \\
 \left(\bv , \varphi \right)|_{t=0} &= \left( \bv_0 , \varphi_0 \right) & \mbox{in } \, \Omega , \label{degequ7}
\end{align}
\end{subequations}
where we set %$\bJ$ such that $\tbj = -\frac{\tilde{\rho}_2 - \tilde{\rho}_1}{2} m(\varphi) \nabla \mu = -\beta \bJ$ with
$\beta = \frac{\tilde{\rho}_2 - \tilde{\rho}_1}{2}$ and $\bJ = -m(\varphi) \nabla \mu$ as indicated above. 
The constitutive relation between density and phase field is given by 
$\rho(\p) = \frac{1}{2}(\tilde{\rho}_1 + \tilde{\rho}_2) + \frac{1}{2} (\tilde{\rho}_2 - \tilde{\rho}_1) \p$
as derived in Abels, Garcke and Gr\"un \cite{AGG12}, where $\tilde{\rho}_i>0$ are the specific constant mass 
densities of the unmixed fluids and $\varphi$ is the difference of the volume fractions of the fluids. 
By introducing $\bJ$, we omitted 
the chemical potential $\mu$ in our equations and we search from now on for unknowns $(\bv,\varphi,\boldsymbol{J})$.
In the above formulation and in the following, we use the abbreviations for space-time cylinders
$Q_{(s,t)}=\Omega \times (s,t)$ and $Q_t = Q_{(0,t)}$ and analogously for the 
boundary $S_{(s,t)}=\partial \Omega \times (s,t)$ and $S_t = S_{(0,t)}$.
Equation \eqref{degequ5} is the no-slip boundary condition for viscous fluids, 
$(\bJ \cdot n)|_{\partial \Omega} = 0$ resulting from $\partial_n \mu |_{\partial \Omega} = 0$ means that there is no mass flux of the components 
through the boundary, and $\partial_n \p |_{\partial \Omega} = 0$ describes a contact angle of $\pi/2$ of the 
diffused interface and the boundary of the domain. 

%%%%%%%%%%%%%%%%%%%%%%%%%%%%%%%%%%%%%%%%%%%%%%%%%%%%%%%%%%%%%%%%%%%%%%%%%%%%%%%%%%%%%%%%%%%%%%%%%%%%%%%%%%
%%%%%%%%%%%%%%%%%%%%%%%%%%%%%%%%%%%%%%%%%%%%%%%%%%%%%%%%%%%%%%%%%%%%%%%%%%%%%%%%%%%%%%%%%%%%%%%%%%%%%%%%%%
%%%%%%%%%%%%%%%%%%%%%%%%%%%%%%%%%%%%%%%%%%%%%%%%%%%%%%%%%%%%%%%%%%%%%%%%%%%%%%%%%%%%%%%%%%%%%%%%%%%%%%%%%%
\subsection{Assumptions and Existence Theorem for Weak Solutions}

In the following we summarize the assumptions needed to formulate the notion of 
a weak solution of \eqref{degequ1}-\eqref{degequ7} and an existence result.

\begin{assumption} \label{assumptions}
We assume that $\Omega \subset \R^d$, $d=2,3$, is a bounded domain with smooth boundary and 
additionally we impose the following conditions.
\begin{enumerate}
 \item We assume $a, \Psi \in C^1(\R)$, $\eta \in C^0(\R)$ and $0 < c_0 \leq a(s), \eta(s) \leq K$ for 
       given constants $c_0,K > 0$.
 \item For the mobility $m$ we assume that 
       \begin{align} \label{assumpdegmob}
         m(s) = \left\{ \begin{array}{cl}
                  1 - s^2 \,, & \mbox{if } \, |s| \leq 1 , \\
                  0 \,, & \mbox{else} .
                 \end{array}
        \right.
       \end{align}
 %\item Additionally we impose the condition that $\lim_{s \rightarrow \pm 1} 
 %      \frac{\Psi''(s)}{|\Psi'(s)|} = +\infty$. % Note: Reference on this item without labels in the next remark!
\end{enumerate}
\end{assumption}
We remark that other mobilities which degenerate linearly at $s = \pm 1$ are possible. The choice 
\eqref{assumpdegmob} typically appears in applications, see Cahn and Taylor \cite{CT94} and 
Hilliard \cite{Hil70}. Other degeneracies can be handled as well but some would need additional
assumptions, see Elliott and Garcke \cite{EG96}. 

We reformulate the model suitably due to the positive coefficient $a(\varphi)$ in the free energy, so that
we can replace the two terms with $a(\varphi)$ in equation \eqref{degequ4} by a single one.
To this end, we introduce the function $A(s) \ZuWeis \int_0^s \sqrt{a(\tau)} \, d\tau$.
Then $A'(s) = \sqrt{a(s)}$ and
\begin{align*}
 - \sqrt{a(\p)} \, \Delta A(\p) = a'(\p) \, \frac{|\nabla \p|^2}{2} - \operatorname{div} \left(a(\p) \, \nabla \p \right)
\end{align*}
resulting from a straightforward calculation. By reparametrizing the potential $\Psi$ through
$\widetilde{\Psi}: \R \to \R$, $\widetilde{\Psi}(r) \hspace*{-1pt}:=\hspace*{-1pt} \Psi(A^{-1}(r))$ we see 
$\Psi'(s) \hspace*{-1pt }= \hspace*{-1pt}\sqrt{a(s)} \widetilde{\Psi}'(A(s))$ and therefore we can replace line \eqref{degequ4} 
with the following one:
\begin{align} \label{replacedegequ4}
 \bJ &= -m(\varphi) \nabla \left( 
      \sqrt{a(\varphi)} \left( \widetilde{\Psi}'(A(\varphi)) - \Delta A(\varphi) \right) 
                           \right).
\end{align}
We also rewrite the free energy with the help of $A$ to 
\begin{align*}
 E_{\mbox{\footnotesize free}}(\p) = \int_\Omega \left( \widetilde{\Psi}(A(\p)) + \frac{|\nabla A(\p)|^2}{2} \right) dx \,.
\end{align*}

\begin{remark}
 With the above notation and with the calculation
 \begin{align*}
  - &\operatorname{div} (a(\p) \nabla \p \otimes \nabla \p) \\
    &= - \operatorname{div} (a(\p) \nabla \p) \nabla \p - a(\p) \nabla \left( \frac{|\nabla \p|^2}{2} \right) \\
    &= - \operatorname{div} (a(\p) \nabla \p) \nabla \p + \nabla (a(\varphi)) \frac{|\nabla \p|^2}{2} 
            - \nabla \left( a(\varphi) \frac{|\nabla \p|^2}{2} \right) \\
    &= \left( - \operatorname{div} (a(\p) \nabla \p) + a'(\varphi) \frac{|\nabla \p|^2}{2} \right) \nabla \varphi 
            - \nabla \left( a(\varphi) \frac{|\nabla \p|^2}{2} \right) \\
    &= - \sqrt{a(\varphi)} \Delta A(\varphi) \nabla \varphi - \nabla \left( a(\varphi) \frac{|\nabla \p|^2}{2} \right)
 \end{align*}
 we rewrite line \eqref{degequ1} with a new pressure $g = p + a(\varphi) \frac{|\nabla \p|^2}{2}$ into:
 \begin{align} 
 \begin{split} \label{replacedegequ1}
  \partial_t (\rho \mathbf{v}) & + \operatorname{div} (\rho \bv \otimes \bv) - \operatorname{div} (2 \eta(\varphi) D \bv)
   + \nabla g + \operatorname{div}(\bv \otimes \beta \bJ) \\
   & = - \sqrt{a(\varphi)} \Delta A(\varphi) \nabla \varphi \,.
 \end{split}
 \end{align}
 We remark that in contrast to the formulation in \cite{ADG12} we do not use the equation 
 for the chemical potential here.
\end{remark}

Now we can define a weak solution of problem \eqref{degequ1}-\eqref{degequ7}. 
\begin{definition} \label{defweaksolution}
 Let $T \in (0,\infty)$, $\bv_0 \in L^2_\sigma(\Omega)$ and $\p_0 \in H^1(\Omega)$ with $|\varphi_0| \leq 1$ 
 almost everywhere in $\Omega$.
 If in addition Assumption \ref{assumptions} holds, we call the triple $(\bv,\p,\bJ)$ 
 with the properties
 \begin{align*}
  & \bv \in BC_w([0,T];L^2_\sigma(\Omega)) \cap L^2(0,T;H_0^1(\Omega)^d) \,, \\
  & \varphi \in BC_w([0,T];H^1(\Omega)) \cap L^2(0,T;H^2(\Omega)) \; \mbox{ with } \;
       |\varphi| \leq 1 \, \mbox{ a.e. in } \, Q_T \,, \\
  & \bJ \in L^2(0,T;L^2(\Omega)^d) \, \mbox{ and} \\
  & \left( \bv,\p \right)|_{t=0} = \left( \bv_0 , \p_0 \right)
 \end{align*}
 a weak solution of \eqref{degequ1}-\eqref{degequ7}
 if the following conditions are satisfied:
\begin{align}  
\begin{split} \label{weakline1} 
  - \left(\rho \bv , \partial_t \bpsi \right)_{Q_T} 
 &+ \left( \operatorname{div}(\rho \bv \otimes \bv) , \bpsi \right)_{Q_T}
 + \left(2 \eta(\p) D\bv , D\bpsi \right)_{Q_T}  \\
 & - \left( (\bv \otimes \beta \bJ) , \nabla \bpsi \right)_{Q_T} 
 = -\left( \sqrt{a(\varphi)} \Delta A(\varphi) \, \nabla \varphi , \bpsi \right)_{Q_T} 
\end{split}
\end{align}
for all $\bpsi \in \left[C_0^\infty(\Omega \times (0,T))\right]^d$ with $\operatorname{div} \bpsi = 0$,
\begin{align}
  -\int_{Q_T} \varphi \, \partial_t \zeta \, dx \, dt 
  + \int_{Q_T} (\bv \cdot \nabla \varphi) \, \zeta \, dx \, dt
   = \int_{Q_T} \mathbf{J} \cdot \nabla \zeta \, dx \, dt \label{weakline2}
\end{align}
for all $\zeta \in C_0^\infty((0,T;C^1(\overline{\Omega}))$ and 
%L^2(0,T;H^1(\Omega)) \cap H^1(0,T;L^2(\Omega))$ and 
\begin{align}
\begin{split} \label{weakline3}
 \int_{Q_T} &\mathbf{J} \cdot \boldsymbol{\eta} \, dx \, dt \\
  &= -\int_{Q_T} \left( \sqrt{a(\varphi)} \left( \widetilde{\Psi}'(A(\varphi)) 
         - \Delta A(\varphi) \right) \right) \di (m(\varphi) \boldsymbol{\eta})\, dx \, dt 
\end{split}
\end{align}
for all $\boldsymbol{\eta} \in L^2(0,T;H^1(\Omega)^d) \cap L^\infty(Q_T)^d$ which fulfill $\boldsymbol{\eta} \cdot n = 0$ on $S_T$.
\end{definition}

\begin{remark}
 The identity \eqref{weakline3} is a weak version of 
 \begin{align*}
   \mathbf{J} = - m(\varphi) \, \nabla \left( \sqrt{a(\varphi)} \left( \widetilde{\Psi}'(A(\varphi)) - \Delta A(\varphi) \right) \right) \,.
 \end{align*}
\end{remark}

Our main result of this work is the following existence theorem for weak solutions on an arbitrary time interval
$[0,T]$, where $T > 0$. 

\begin{theorem} \label{theo:existenceweaksolution}
 Let Assumption \ref{assumptions} hold, $\bv_0 \in L^2_\sigma(\Omega)$ and $\varphi_0 \in H^1(\Omega)$
 with $|\varphi_0| \leq 1$ almost everywhere in $\Omega$. 
 Then there exists a weak solution $(\bv,\varphi,\bJ)$ of \eqref{degequ1}-\eqref{degequ7} 
 in the sense of Definition \ref{defweaksolution}. Moreover for some $\widehat{\bJ} \in L^2(Q_T)$ it holds that
 $\bJ = \sqrt{m(\varphi)} \widehat{\bJ}$ and 
 \begin{align} 
 \begin{split} \label{weakline5}
  E_{\mbox{\footnotesize tot}}(\p(t),\bv(t)) &+ \int_{Q_{(s,t)}} 2 \eta(\p) \, |D\bv|^2 \, dx \, d\tau 
        + \int_{Q_{(s,t)}} |\widehat{\bJ}|^2 \, dx \, d\tau \\
   &\leq E_{\mbox{\footnotesize tot}}(\p(s),\bv(s)) 
 \end{split}
 \end{align}
 for all $t \in [s,T)$ and almost all $s \in [0,T)$ including $s=0$.
 The total energy $E_{\mbox{\footnotesize tot}}$ is the sum of the kinetic and the free energy, cf. \eqref{totalenergy}.
 In particular, $\bJ = 0$ a.e. on the set $\{|\varphi| = 1\}$.
\end{theorem}

The proof of the theorem will be done in the next two subsections. But first of all we consider a special 
case which can then be excluded in the following proof. Due to $|\varphi_0| \leq 1$ a.e. in $\Omega$ we 
note that $\int_\Omega \hspace*{-14pt}{-}\hspace*{5pt} \varphi_0 \, dx \in [-1,1]$. In the situation
where $\int_\Omega \hspace*{-14pt}{-}\hspace*{5pt} \varphi_0 \, dx = 1$ we can then conclude that 
$\varphi_0 \equiv 1$ a.e. in $\Omega$ and can give the solution at once. In fact, here we set $\varphi \equiv 1$, 
$\bJ \equiv 0$ and let $\bv$ be the weak solution of the incompressible Navier-Stokes equations
without coupling to the Cahn-Hilliard equation, where $\rho$ and $\eta$ are constants. 
The situation where $\int_\Omega \hspace*{-14pt}{-}\hspace*{5pt} \varphi_0 \, dx = -1$ can 
be handled analogously.

With this observation we can assume in the following that 
\begin{align*} %\label{assumpmeanvalue}
 \int_\Omega \hspace*{-13pt}{-}\hspace*{4pt} \varphi_0 \, dx \in (-1,1) \,,
\end{align*}
which will be needed for the reference to the previous existence result of the authors \cite{ADG12}
and for the proof of Lemma \ref{lem:energyestimate}, $(iii)$. 

%%%%%%%%%%%%%%%%%%%%%%%%%%%%%%%%%%%%%%%%%%%%%%%%%%%%%%%%%%%%%%%%%%%%%%%%%%%%%%%%%%%%%%%%%%%%%%%%%%%%%%%%%%
%%%%%%%%%%%%%%%%%%%%%%%%%%%%%%%%%%%%%%%%%%%%%%%%%%%%%%%%%%%%%%%%%%%%%%%%%%%%%%%%%%%%%%%%%%%%%%%%%%%%%%%%%%
%%%%%%%%%%%%%%%%%%%%%%%%%%%%%%%%%%%%%%%%%%%%%%%%%%%%%%%%%%%%%%%%%%%%%%%%%%%%%%%%%%%%%%%%%%%%%%%%%%%%%%%%%%
\subsection{Approximation and Energy Estimates}

In the following we substitute problem \eqref{degequ1}-\eqref{degequ7} by an approximation
with positive mobility and a singular homogeneous free energy density, which can be solved 
with the result from the authors in \cite{ADG12}. For the weak solutions of the approximation
we then derive energy estimates.

First we approximate the degenerate mobility $m$ by a strictly positive $m_\varepsilon$ as
\begin{eqnarray*}
 m_\varepsilon(s) := \left\{ \begin{array}{cl}
                              m(-1+\varepsilon) & \mbox{for } \, s \leq -1 + \varepsilon \,, \\
                              m(s) & \mbox{for } \, |s| < 1-\varepsilon \,, \\
                              m(1-\varepsilon) & \mbox{for } \, s \geq 1 - \varepsilon \,.
                             \end{array}
                     \right.
\end{eqnarray*}
In addition we use a singular homogeneous free energy density $\Psi_\varepsilon$ given by 
\begin{align*}
 \Psi_\varepsilon(s) &:=  \Psi(s) + \varepsilon \Psi_{\mbox{\footnotesize ln}}(s) \,, \; \mbox{ where } \\
 \Psi_{\mbox{\footnotesize ln}}(s) &:= (1+s)\ln(1+s) + (1-s)\ln(1-s)\,.
\end{align*}
Then $\Psi_\varepsilon \in C([-1,1]) \cap C^2((-1,1))$ fulfills the assumptions on the homogeneous free energy as in 
Abels, Depner and Garcke \cite{ADG12}, which were given by
\begin{align*}
 \lim_{s \to \pm 1} \Psi_\e'(s) = \pm \infty \,, \quad 
 \Psi_\e''(s) \geq \kappa \; \mbox{ for some } \; \kappa \in \R \; \mbox{  and } \;
 \lim_{s \to \pm 1} \frac{\Psi_\e''(s)}{\Psi_\e'(s)} = + \infty .
\end{align*}
To deal with the positive coefficient $a(\varphi)$, we set similarly as above
$\widetilde{\Psi}_{\mbox{\footnotesize ln}}(r) := \Psi_{\mbox{\footnotesize ln}}(A^{-1}(r))$ and 
$\widetilde{\Psi}_\varepsilon(r) := \Psi_\varepsilon(A^{-1}(r))$ for $r \in [a,b] := A([-1,1])$.

Now we replace $m$ by $m_\e$ and $\Psi$ by $\Psi_\e$ and consider
the following approximate problem, this time for unknowns $(\bv,\varphi,\mu)$:
\begin{subequations}
\begin{align} 
  \partial_t (\rho \mathbf{v}) &+ \di \left(\rho \bv \otimes \bv \right) 
   - \di \left(2 \eta(\varphi) D\bv \right) 
   + \nabla g \nonumber \\  
  & + \di \left(\bv \otimes \beta m_\e(\varphi) \nabla \mu \right) 
  = - \sqrt{a(\varphi)} \Delta A(\varphi) \nabla \varphi & \mbox{in } \, Q_T , \label{approxequ1} \\
  \di \mathbf{v} &= 0 & \mbox{in } \, Q_T ,  \label{approxequ2} \\
  \partial_t \varphi + \mathbf{v} \cdot \nabla \varphi &= \di (m_\e(\varphi) \nabla \mu) & \mbox{in } \, Q_T , \label{approxequ3} \\
  \mu &= \sqrt{a(\varphi)} \left( \widetilde{\Psi}_\e'(A(\varphi)) - \Delta A(\varphi) \right)  & \mbox{in } \, Q_T , \label{approxequ5} \\
  \bv|_{\partial \Omega} &= 0 & \mbox{on } \, S_T , \label{approxequ6} \\
  \partial_n \varphi|_{\partial \Omega} &= \partial_n \mu|_{\partial \Omega} = 0 & \mbox{on } \, S_T , \label{approxequ7} \\
  \left(\bv , \varphi \right)|_{t=0} &= \left( \bv_0 , \varphi_0 \right) & \mbox{in } \, \Omega . \label{approxequ8}
\end{align}
\end{subequations}
From \cite{ADG12} we get the existence of a weak solution $(\bv_\e,\varphi_\e,\mu_\e)$ with the properties
\begin{align*}
  & \bv_\e \in BC_w([0,T];L^2_\sigma(\Omega)) \cap L^2(0,T;H_0^1(\Omega)^d) \,, \\
  & \p_\e \in BC_w([0,T];H^1(\Omega)) \cap L^2(0,T;H^2(\Omega)) \,, \; \
        \Psi_\e'(\p_\e) \in L^2(0,T;L^2(\Omega)) \,, \\
  & \mu_\e \in L^2(0,T;H^1(\Omega)) \, \mbox{ and} \\
  & \left.\left( \bv_\e,\p_\e \right)\right|_{t=0} = \left( \bv_0 , \p_0 \right)
\end{align*}
in the following sense:
 \begin{align}  
 \begin{split} \label{approxweak1} 
   - \left(\rho_\e \bv_\e , \partial_t \bpsi \right)_{Q_T} 
  &+ \left( \operatorname{div}(\rho_\e \bv_\e \otimes \bv_\e) , \bpsi \right)_{Q_T}
  + \left(2 \eta(\p_\e) D\bv_\e , D\bpsi \right)_{Q_T} \\
  & - \left( (\bv_\e \otimes \beta m_\e(\varphi_\e) \nabla \mu_\e ) , \nabla \bpsi \right)_{Q_T} 
  = \left( \mu_\e \nabla \varphi_\e , \bpsi \right)_{Q_T} 
 \end{split}
 \end{align}
 for all $\bpsi \in \left[C_0^\infty(\Omega \times (0,T))\right]^d$ with $\operatorname{div} \bpsi = 0$,
 \begin{align} 
  - \left(\p_\e , \partial_t \zeta \right)_{Q_T} 
  + \left( \bv_\e \cdot \nabla \p_\e , \zeta \right)_{Q_T}
  &= - \left(m_\e(\p_\e) \nabla \mu_\e , \nabla \zeta \right)_{Q_T}  \label{approxweak2} 
 \end{align}
 for all $\zeta \in C_0^\infty((0,T);C^1(\overline{\Omega}))$ and
 \begin{align}
   \mu_\e = \sqrt{a(\varphi_\e)} \left( \widetilde{\Psi}_\e'(A(\varphi_\e)) - \Delta A(\varphi_\e) \right)
     \; \mbox{ almost everywhere in } \, Q_T . \label{approxweak3}
 \end{align}
 Moreover,
 \begin{align} 
 \begin{split} \label{approxweak4}
  E_{\mbox{\footnotesize tot}}(\p_\e(t),\bv_\e(t)) &+ \int_{Q_{(s,t)}} 2 \eta(\p_\e) \, |D\bv_\e|^2 \, dx \,d\tau \\
        &+ \int_{Q_{(s,t)}} m_\e(\p_\e) |\nabla \mu_\e|^2 \, dx \, d\tau 
   \,\leq\, E_{\mbox{\footnotesize tot}}(\p_\e(s),\bv_\e(s)) 
 \end{split}
 \end{align}
 for all $t \in [s,T)$ and almost all $s \in [0,T)$ has to hold (including $s=0$).

Herein $\rho_\e$ is given as 
$\rho_\e = \frac{1}{2}(\tilde{\rho}_1 + \tilde{\rho}_2) + \frac{1}{2} (\tilde{\rho}_2 - \tilde{\rho}_1) \p_\e$.
Note that due to the singular homogeneous potential $\Psi_\e$ we have 
$|\varphi_\e| < 1$ almost everywhere.

\begin{remark} \label{rem:termsagree}
 Note that equation \eqref{approxweak1} can be rewritten with the help of the identity
 \begin{align*}
  \left( \mu_\e \nabla \varphi_\e , \bpsi \right)_{Q_T} 
    &= - \left( \sqrt{a(\varphi_\e)} \Delta A(\varphi_\e) \nabla \varphi_\e , \bpsi \right)_{Q_T}.
 \end{align*}
 This can be seen by testing \eqref{approxweak3} with $\nabla \varphi_\e \cdot \bpsi$ and noting
 that $\bpsi$ is divergence free.
\end{remark}

For the weak solution $(\bv_\e,\varphi_\e,\mu_\e)$ we get the following energy estimates:

\begin{lemma} \label{lem:energyestimate}
 For a weak solution $(\bv_\e,\varphi_\e,\mu_\e)$ of problem \eqref{approxequ1}-\eqref{approxequ8} we have the following 
 energy estimates:
 \begin{align*}
  (i) & \quad \sup_{0 \leq t \leq T} \int_\Omega \left( \rho_\e(t) \frac{\;|\bv_\e(t)|^2}{2} 
                           + \frac{1}{2} |\nabla \varphi_\e(t)|^2 + \Psi_\e(\varphi_\e(t)) \right) dx \\
          & \quad \; + \int_{Q_T} 2 \eta(\varphi_\e) |D \bv_\e|^2 \, dx \, dt 
                     + \int_{Q_T} m_\e(\varphi_\e) |\nabla \mu_\e|^2 \, dx \, dt
           \,\leq\, C \,, \\
  (ii) & \quad \sup_{0 \leq t \leq T} \int_\Omega \hspace*{-2pt} G_\e(\varphi_\e(t)) \, dx 
           + \int_{Q_T} |\Delta A(\varphi_\e)|^2  dx \, dt \,\leq\, C \,, \\
  (iii) & \quad \e^3 \int_{Q_T} |\Psi_{\mbox{\footnotesize ln}}'(\varphi_\e)|^2 \, dx \, dt \,\leq\, C \,, \\
  (iv) & \quad \int_{Q_T} |\widehat{\bJ}_\e|^2 \, dx \, dt \,\leq\, C \,, 
           \; \mbox{ where } \, \widehat{\bJ}_\e = - \sqrt{m_\e(\varphi_\e)} \,\nabla \mu_\e \,.
 \end{align*}
 Here $G_\e$ is a non-negative function defined by $G_\e(0) = G_\e'(0) = 0$ and 
 $G_\e''(s) = \frac{1}{m_\e(s)} \sqrt{a(s)}$ for $s \in [-1,1]$.
\end{lemma}
\begin{proof}
ad $(i)$: This follows directly from the estimate \eqref{approxweak4} derived in the work of Abels, Depner and Garcke \cite{ADG12}. 
We just note that for the estimate of $\nabla \varphi_\e$ we use $\nabla A(\varphi_\e) = \sqrt{a(\varphi_\e)} \nabla \varphi_\e$
and the fact that $a$ is bounded from below by a positive constant due to Assumption \ref{assumptions}.

ad $(ii)$: From line \eqref{approxweak2} we get that $\partial_t \varphi_\e \in L^2(0,T;\left( H^1(\Omega) \right)')$,
since $\nabla \mu_\e \in L^2(Q_T)$ and $\bv \cdot \nabla \varphi = \operatorname{div} (\bv \,\varphi)$ with
$\bv \, \varphi \in L^2(Q_T)$.
Then we derive for a function $\zeta \in L^2(0,T;H^2(\Omega))$
 the weak formulation
 \begin{align}
 \begin{split} \label{lab1}
  \int_0^t & \langle \partial_t \varphi_\e , \zeta \rangle \, d\tau 
    + \int_{Q_t} \bv_\e \cdot \nabla \varphi_\e \zeta \, dx \, d\tau \\
    &= - \int_{Q_t} m_\e(\varphi_\e) \nabla \mu_\e \cdot \nabla \zeta \, dx \, d\tau  \\
    &= \int_{Q_t} \sqrt{a(\varphi_\e)} \left( \widetilde{\Psi}_\e'(A(\varphi_\e)) 
                                                   - \Delta A(\varphi_\e) \right) \di (m_\e(\varphi_\e) \nabla \zeta) \, dx \, d\tau ,
 \end{split}
 \end{align}
where we additionally used \eqref{approxweak3} to express $\mu_\e$. 
Now we set as test function $\zeta = G_\e'(\varphi_\e)$, where $G_\e$ is defined by $G_\e(0) = G_\e'(0) = 0$ and 
$G_\e''(s) = \frac{1}{m_\e(s)} A'(s)$ for $s \in [-1,1]$. Note that $G_\e$ is a non-negative function, which 
can be seen from the representation
$G_\e(s) = \int_0^s \left( \int_0^r \frac{1}{m_\e(\tau)} A'(\tau) \, d\tau \right) dr$. 
With $\zeta = G_\e'(\varphi_\e)$ it holds that
\begin{align*}
 \nabla \zeta &= G''_\e(\varphi_\e) \nabla \varphi_\e  
              = \frac{1}{m_\e(\varphi_\e)} \nabla \left( A(\varphi_\e) \right) \; \mbox{ and therefore } \; \\
 \di \left(m_\e(\varphi_\e) \nabla \zeta \right) &= \Delta \left( A(\varphi_\e) \right) .
\end{align*}
Hence we derive 
\begin{align}
\begin{split} \label{lab2}
 \int_0^t &\langle \partial_t \varphi_\e , G'_\e(\varphi_\e) \rangle d\tau 
    + \int_{Q_t} \bv_\e \cdot \nabla \varphi_\e G'_\e(\varphi_\e) \, dx \, d\tau \\
    &= \int_{Q_t} \sqrt{a(\varphi_\e)}\left( \widetilde{\Psi}_\e'(A(\varphi_\e)) - \Delta A(\varphi_\e) \right) \Delta A(\varphi_\e) \, dx \, d\tau \\
    &= \int_{Q_t} \Psi_\e'(\varphi_\e) \Delta A(\varphi_\e) \, dx \, d\tau
       - \int_{Q_t} \sqrt{a(\varphi_\e)} \, |\Delta A(\varphi_\e)|^2 \, dx \, d\tau \,.
\end{split}
\end{align}
With this notation we deduce
\begin{align*}
 \int_0^t \langle \partial_t \varphi_\e , G'_\e(\varphi_\e) \rangle dt 
   &= \int_\Omega G_\e(\varphi(t)) \, dx - \int_\Omega G_\e(\varphi_0) \, dx \quad \mbox{ and } \\
 \int_{Q_t} \bv_\e \cdot \nabla \varphi_\e G'_\e(\varphi_\e) \, dx \, dt  
   &= \int_{Q_t} \bv_\e \cdot \nabla \left(G_\e(\varphi_\e)\right) dx \, dt \\
   &= -\int_{Q_t} \di \bv_\e \, G_\e(\varphi_\e) \, dx \, dt = 0 \,.
\end{align*}
{\allowdisplaybreaks
For the first term on the right side of \eqref{lab2} we observe
\begin{align*}
 \int_{Q_t} &\Psi_\e'(\varphi_\e) \Delta A(\varphi_\e) \, dx \, d\tau \\
  &= \int_{Q_t} \Psi'(\varphi_\e) \Delta A(\varphi_\e) \, dx \, d\tau
     + \e \int_{Q_t} \Psi_{\mbox{\footnotesize ln}}'(\varphi_\e) \Delta A(\varphi_\e) \, dx \, d\tau \\
  & \leq  -\int_{Q_t} \Psi''(\varphi_\e) \nabla \varphi_\e \cdot \nabla A(\varphi_\e) \, dx \, dt \\
  &= -\int_{Q_t} \Psi''(\varphi_\e) \sqrt{a(\varphi_\e)} |\nabla \varphi_\e|^2 \, dx \, dt .
\end{align*}}
Herein the estimate
\begin{align*}
 \int_{Q_t} \Psi_{\mbox{\footnotesize ln}}'(\varphi_\e) \Delta A(\varphi_\e) \, dx \, d\tau &\leq 0 
\end{align*}
for the logarithmic 
part of the homogeneous free energy density is derived as follows. With an 
approximation of $\varphi_\e$ by $\varphi_\e^\alpha = \alpha \varphi_\e$ for $0<\alpha<1$ we have 
that $|\varphi_\e^\alpha| < \alpha < 1$ and therefore
\begin{align*}
 \int_{Q_t} \Psi_{\mbox{\footnotesize ln}}'(\varphi_\e^\alpha) \Delta A(\varphi_\e^\alpha) \, dx \, d\tau
&= -\int_{Q_t} \Psi_{\mbox{\footnotesize ln}}''(\varphi_\e^\alpha) \nabla \varphi_\e^\alpha \cdot \nabla A(\varphi_\e^\alpha) \, dx \, d\tau
\leq 0 \,,
\end{align*}
where we used integration by parts. 
To pass to the limit for $\alpha \nearrow 1$ in the left side we observe that 
$\varphi_\e^\alpha \to \varphi_\e$ in $L^2(0,T;H^2(\Omega))$. Hence together with the bound
$|\Psi_{\mbox{\footnotesize ln}}'(\varphi_\e^\alpha)| \leq |\Psi_{\mbox{\footnotesize ln}}'(\varphi_\e)|$ 
we can use Lebesgue's dominated convergence theorem to conclude
\begin{align*}
 \int_{Q_t} \Psi_{\mbox{\footnotesize ln}}'(\varphi_\e^\alpha) \Delta A(\varphi_\e^\alpha) \, dx \, d\tau 
  \longrightarrow \int_{Q_t} \Psi_{\mbox{\footnotesize ln}}'(\varphi_\e) \Delta A(\varphi_\e) \, dx \, d\tau 
 \; \mbox{ for } \; \alpha \nearrow 1.
\end{align*}
With the bound from below $a(s) \geq c_0 > 0$ from Assumption \ref{assumptions} we derived therefore
\begin{align*}
 \int_\Omega &G_\e(\varphi(t)) \, dx 
  + \int_{Q_t} |\Delta A(\varphi_\e)|^2 \, dx \, d\tau \\
 &\leq C \left( \int_\Omega G_\e(\varphi_0) \, dx 
   + \int_{Q_t} \Psi''(\varphi_\e) \sqrt{a(\varphi_\e)} \, |\nabla \varphi_\e|^2 \, dx \, d\tau \right) .
\end{align*}
Now we use  $m_\e(\tau) \geq m(\tau)$ to observe the inequality
\begin{align*}
 G_\e(s) &= \int_0^s \bigg( \int_0^r \frac{1}{m_\e(\tau)} \underbrace{A'(\tau)}_{=\sqrt{a(\tau)}} \, d\tau \bigg) dr \\
         &\leq \int_0^s \left( \int_0^r \frac{1}{m(\tau)} \sqrt{a(\tau)} \, d\tau \right) dr
         =: G(s) \; \mbox{ for s $\in (-1,1)$}.
\end{align*}
Due to the special choice of the degenerate mobility $m$ in \eqref{assumpdegmob} 
we conclude that $G$ can be extended continuously to the closed interval $[-1,1]$ and that therefore 
the integral $\int_\Omega G(\varphi_0) \, dx$ and in particular the integral $\int_\Omega G_\e(\varphi_0) \, dx$ 
is bounded.

Moreover, since $\Psi''(s)$ is bounded in $|s| \leq 1$ 
and since we estimated $\int_\Omega |\nabla \varphi_\e(t)|^2 \, dx$ in $(i)$, we proved $(ii)$. 

ad $(iii)$: To show this estimate we will argue similarly as in the time-discrete situation of Lemma 4.2 
in Abels, Depner and Garcke \cite{ADG12}. 
We multiply equation \eqref{approxweak3} with $P_0 \varphi_\e$, integrate 
over $\Omega$ and get almost everywhere in $t$ the identity
\begin{align} \label{P0ident}
\begin{split}
 \int_{\Omega} \mu_\e P_0 \varphi_\e \, dx &= \int_{\Omega} \Psi'(\varphi_\e) P_0 \varphi_\e \, dx 
   + \e \int_{\Omega} \Psi_{\mbox{\footnotesize ln}}'(\varphi_\e) P_0 \varphi_\e \, dx \\
   & \quad - \int_{\Omega} \sqrt{a(\varphi_\e)} \Delta A(\varphi_\e) P_0 \varphi_\e \, dx .
\end{split}
\end{align}
By using in identity \eqref{approxweak2} a test function which depends only on time $t$ and not on $x \in \Omega$, 
we derive the fact that $(\varphi_\e)_\Omega = (\varphi_0)_\Omega$ and by assumption this number lies in 
$(-1+\alpha, 1-\alpha)$ for a small $\alpha > 0$.
In addition with the property $\lim_{s \to \pm 1} \Psi_{\mbox{\footnotesize ln}}'(s) = \pm \infty$
we can show the inequality $\Psi_{\mbox{\footnotesize ln}}'(s)(s - (\varphi_0)_\Omega) 
\geq C_\alpha |\Psi_{\mbox{\footnotesize ln}}'(s)| - c_\alpha$ in three steps in the intervals
$[-1,-1+\tfrac{\alpha}{2}]$, $[-1+\tfrac{\alpha}{2}, 1-\tfrac{\alpha}{2}]$ and $[1-\tfrac{\alpha}{2},1]$ successively.
Altogether this leads to the following estimate:
\begin{align} \label{P0est}
 \e \int_{\Omega} |\Psi_{\mbox{\footnotesize ln}}'(\varphi_\e)| \, dx 
   &\leq  C \left( \e \int_{\Omega} \Psi_{\mbox{\footnotesize ln}}'(\varphi_\e) P_0 \varphi_\e \, dx + 1 \right) .
\end{align}
We observe the fact that $\int_{\Omega} \mu_\e P_0 \varphi_\e \, dx = \int_{\Omega} (P_0 \mu_\e ) \varphi_\e \, dx$ 
and due to integration by parts
\begin{align*}
  - \int_{\Omega} & \sqrt{a(\varphi_\e)} \Delta A(\varphi_\e) P_0 \varphi_\e \, dx \\
    &= \int_\Omega \sqrt{a(\varphi_\e)} \nabla A(\varphi_\e) \cdot \nabla \varphi_\e \, dx 
      + \int_\Omega \frac{1}{2} a(\varphi_\e)^{-\frac{1}{2}} \nabla \varphi_\e \cdot \nabla A(\varphi_\e) P_0 \varphi_\e \, dx \\
    &= \int_\Omega a(\varphi_\e) |\nabla \varphi_\e|^2 \, dx 
        + \int_\Omega \frac{1}{2} P_0 \varphi_\e |\nabla \varphi_\e|^2 \, dx \,.
\end{align*}
Combining estimate \eqref{P0est} with identity \eqref{P0ident} we are led to
\begin{align*}
 \e \int_{\Omega} |\Psi_{\mbox{\footnotesize ln}}'(\varphi_\e)| \, dx 
  &\leq C \left( \int_{\Omega} |(P_0 \mu_\e) \varphi_\e| \, dx + \int_{\Omega} |\Psi'(\varphi_\e) P_0 \varphi_\e |\, dx \right. \\
  & \quad  + \left. \int_{\Omega} |\sqrt{a(\varphi_\e)} \Delta A(\varphi_\e) P_0 \varphi_\e| \, dx + 1 \right) \\
  &\leq C \left( \|P_0 \mu_\e\|_{L^2(\Omega)} + \|\nabla \varphi_\e\|_{L^2(\Omega)} + 1 \right) \\
  & \leq C \left( \|\nabla \mu_\e \|_{L^2(\Omega)} + 1 \right) .
\end{align*}
In the last two lines we have used in particular the facts that $\varphi_\e$ is bounded between $-1$ and $1$, 
that $\Psi'$ is continuous,
the energy estimate from $(ii)$ for $\sup_{0 \leq t \leq T} \|\nabla \varphi_\e\|_{L^2(\Omega)}$ 
and the Poincar\'{e} inequality for functions with mean value zero. 

With the last inequality we can estimate the integral of $\mu_\e$ by simply integrating 
identity \eqref{approxweak3} over $\Omega$:
\begin{align*}
 \left| \int_{\Omega} \mu_\e \, dx \right| 
  &\leq \int_{\Omega} |\Psi'(\varphi_\e)| \,dx + \e \int_{\Omega} |\Psi_{\mbox{\footnotesize ln}}'(\varphi_\e)| \, dx
   + \left| \int_{\Omega} \sqrt{a(\varphi_\e)} \Delta A(\varphi_\e) \, dx \right| \\
  &\leq C \left( \|\nabla \mu_\e \|_{L^2(\Omega)} + 1 \right) ,
\end{align*}
where we used similarly as above integration by parts for the integral over $\sqrt{a(\varphi_\e)} \Delta A(\varphi_\e)$.
By the splitting of $\mu_\e$ into $\mu_\e = P_0 \mu_\e + (\mu_\e)_\Omega$ we arrive at
\begin{align*}
 \|\mu_\e\|_{L^2(\Omega)}^2 &\leq C \left( \|\nabla \mu_\e\|_{L^2(\Omega)}^2 + 1 \right) .
\end{align*}
Then, again from identity \eqref{approxweak3}, we derive
\begin{align*}
 \e^2 |\Psi_{\mbox{\footnotesize ln}}'(\varphi_\e)|^2 &\leq 
  C \left( |\mu_\e|^2 + |\Delta A(\varphi_\e)|^2 + 1 \right)
\end{align*}
and together with the last estimates and an additional integration over time $t$ this leads to
\begin{align*}
 \e^2 \|\Psi_{\mbox{\footnotesize ln}}'(\varphi_\e)\|_{L^2(Q_T)}^2 &\leq 
   C \left( \|\nabla \mu_\e \|^2_{L^2(Q_T)} + 1 \right) . 
\end{align*}
Note that we used the bound $\|\Delta A(\varphi_\e)\|_{L^2(Q_T)} \leq C$ from $(ii)$. 
Furthermore, due to the bounds in $(i)$, we see $\e \|\nabla \mu_\e\|^2_{L^2(Q_T)} \leq C$ since 
$m_\e(s) \geq \e$ for $|s| \leq 1$ and therefore we arrive at
\begin{align*}
 \e^3 \|\Psi_{\mbox{\footnotesize ln}}'(\varphi_\e)\|_{L^2(Q_T)}^2 &\leq C \,.
\end{align*}
ad $(iv)$: This follows directly from $(i)$. 
\end{proof}

\subsection{Passing to the limit in the Approximation}

In this subsection we use the energy estimates to get weak limits for the sequences
$(\bv_\e,\varphi_\e,\bJ_\e)$, where $\bJ_\e = \sqrt{m_\e(\varphi_\e)} \, \widehat{\bJ}_\e \left( = -m_\e(\varphi_\e) \nabla \mu_\e \right)$. 
With some subtle arguments we extend the 
weak convergences to strong ones, so that we are able to pass to the limit for 
$\e \to 0$ in the equations \eqref{approxweak1}-\eqref{approxweak3} to recover
the identities \eqref{weakline1}-\eqref{weakline3} in the definition of the weak
solution for the main problem \eqref{degequ1}-\eqref{degequ7}.

Using the energy estimates in Lemma \ref{lem:energyestimate}, we can pass to a subsequence to get
\begin{align*}
 \bv_\e \rightharpoonup \bv & \;\mbox{ in } \; L^2(0,T;H^1(\Omega)^d) , \\
 \varphi_\e \rightharpoonup \varphi & \; \mbox{ in } \; L^2(0,T;H^1(\Omega)) ,\\
 \widehat{\bJ}_\e \rightharpoonup \widehat{\bJ} & \; \mbox{ in } \; L^2(0,T;L^2(\Omega)^d)  \;\mbox{ and} \\
 \bJ_\e \rightharpoonup \bJ & \; \mbox{ in } \; L^2(0,T;L^2(\Omega)^d)
\end{align*}
for $\bv \in L^2(0,T;H^1(\Omega)^d) \cap L^\infty(0,T;L^2_\sigma(\Omega))$, 
$\varphi \in L^\infty(0,T;H^1(\Omega))$ and $\widehat{\bJ}, \bJ \in L^2(0,T;L^2(\Omega)^d)$.
Here and in the following all limits are meant to be for suitable subsequences $\e_k \to 0$ for 
$k \to \infty$. 

With the notation $\bJ_\e = -m_\e(\varphi_\e) \nabla \mu_\e$ the weak solution of problem \eqref{approxequ1}-\eqref{approxequ8}
fulfills the following equations:
\begin{align}  
\begin{split} \label{approxweakline1} 
  - \big(\rho_\e \bv_\e , &\partial_t \bpsi \big)_{Q_T} 
 + \left( \operatorname{div}(\rho_\e \bv_\e \otimes \bv_\e) , \bpsi \right)_{Q_T}
 + \left(2 \eta(\p_\e) D\bv_\e , D\bpsi \right)_{Q_T} \\
 & - \left( (\bv_\e \otimes \beta \bJ_\e) , \nabla \bpsi \right)_{Q_T} 
 = -\left( \sqrt{a(\varphi_\e)} \Delta A(\varphi_\e) \, \nabla \varphi_\e , \bpsi \right)_{Q_T} 
\end{split}
\end{align}
for all $\bpsi \in \left[C_0^\infty(\Omega \times (0,T))\right]^d$ with $\operatorname{div} \bpsi = 0$,
\begin{align}
  -\int_{Q_T} \varphi_\e \, \partial_t \zeta \, dx \, dt 
  + \int_{Q_T} (\bv_\e \cdot \nabla \varphi_\e) \, \zeta \, dx \, dt 
   = \int_{Q_T} \bJ_\e \cdot \nabla \zeta \, dx \, dt \label{approxweakline2}
\end{align}
for all $\zeta \in C_0^\infty((0,T;C^1(\overline{\Omega}))$ and 
\begin{align}
\begin{split} \label{approxweakline3}
 \int_{Q_T} & \bJ_\e \cdot \boldsymbol{\eta} \, dx \, dt \\ 
 &= -\int_{Q_T} \left( \Psi_\e'(\varphi_\e) - \sqrt{a(\varphi_\e)}\Delta A(\varphi_\e) \right) \di (m_\e(\varphi_\e) \boldsymbol{\eta}) \, dx \, dt
\end{split}
\end{align}
for all $\boldsymbol{\eta} \in L^2(0,T;H^1(\Omega)^d) \cap L^\infty(Q_T)^d$ with $\boldsymbol{\eta} \cdot n = 0$ on $S_T$. 
For the last line we used that for functions $\boldsymbol{\eta}$ with $\boldsymbol{\eta} \cdot n = 0$ on $S_T$ it holds
\begin{align*}
 \int_{Q_T} \bJ_\e \cdot \boldsymbol{\eta} \, dx \, dt
   &= \int_{Q_T} \nabla \mu_\e \cdot m_\e(\varphi_\e) \boldsymbol{\eta} \, dx \, dt
   = -\int_{Q_T} \mu_\e \di(m_\e(\varphi_\e) \boldsymbol{\eta}) \, dx \, dt \\
   &= -\int_{Q_T} \left(\Psi_\e'(\varphi_\e) - \sqrt{a(\varphi_\e)}\Delta A(\varphi_\e) \right) \di(m_\e(\varphi_\e) \boldsymbol{\eta}) \, dx \, dt \,.
\end{align*}
Now we want to pass to the limit $\e \to 0$ in the above equations to achieve finally the weak formulation \eqref{weakline1}-\eqref{weakline3}.

For the convergence in identity \eqref{approxweakline1} we first note that
\begin{align*}
 \partial_t \varphi_\e &\; \mbox{ is bounded in } \; L^2(0,T;\left(H^1(\Omega)\right)') \; \mbox{ and} \\
 \varphi_\e & \;\mbox{ is bounded in } \; L^\infty(0,T;H^1(\Omega)) \,.
\end{align*}
Therefore we can deduce from the Lemma of Aubins-Lions \eqref{eq:AubinLions} the 
strong convergence
\begin{align*}
 \varphi_\e \to \varphi \; \mbox{ in } \; L^2(0,T;L^2(\Omega))
\end{align*}
and $\varphi_\e \to \varphi$ pointwise almost everywhere in $Q_T$. 

From the bound of $\Delta A(\varphi_\e)$ in $L^2(Q_T)$ and from 
\begin{align*}
  \nabla A(\varphi_\e) \cdot n &= \sqrt{a(\varphi_\e)} \nabla \varphi_\e \cdot n = 0 \; \mbox{ on } \, S_T, 
\end{align*}
we 
get from elliptic regularity theory the bound
\begin{align*}
 \|A(\varphi_\e)\|_{L^2(0,T;H^2(\Omega))} & \leq C \,.
\end{align*}
This yields 
\begin{align*}
 A(\varphi_\e) \rightharpoonup g \; \mbox{ in } \; L^2(0,T;H^2(\Omega))
\end{align*} 
at first for some $g \in L^2(0,T;H^2(\Omega))$, but then, due to the weak 
convergence $\nabla \varphi_\e \rightharpoonup \nabla \varphi$ in $L^2(0,T;L^2(\Omega))$
and due to the pointwise almost everywhere convergence $a(\varphi_\e) \to a(\varphi)$ in $Q_T$
we can identify $g$ with $A(\varphi)$ to get
\begin{align*}
 A(\varphi_\e) \rightharpoonup A(\varphi) \; \mbox{ in } \; L^2(0,T;H^2(\Omega)) \,.
\end{align*}
The next step is to strengthen the convergence of $\nabla \varphi_\e$ in $L^2(Q_T)$. 
To this end, we remark that by definition $A$ is Lipschitz-continuous with
\begin{align*}
 |A(r) - A(s)| \leq \left| \int_s^r \sqrt{a(\tau)} \, d\tau \right| \leq C |r-s| \,.
\end{align*}
Furthermore from the bound of $\partial_t \varphi_\e$ in $L^2(0,T;\left(H^1(\Omega)\right)')$
we get with the notation $\varphi_\e(.+h)$ for a shift in time 
\begin{align*}
 \| \varphi_\e(.+h) - \varphi_\e \|_{L^2(0,T-h;\left(H^1(\Omega)\right)')} &\leq Ch \,,
\end{align*}
which leads to the estimate
\begin{align*}
  \lefteqn{\| A(\varphi_\e(.+h)) - A(\varphi_\e) \|_{L^2(0,T-h;\left(H^1(\Omega)\right)')}}  \\
  & \; \leq C \| \varphi_\e(.+h) - \varphi_\e \|_{L^2(0,T-h;\left(H^1(\Omega)\right)')} \\
  & \; \leq C h \longrightarrow  0 \quad \mbox{as } \; h \to 0 \,.
\end{align*}
Together with the bound of $A(\varphi_\e)$ in $L^2(0,T;H^2(\Omega))$ we can use
a theorem of Simon \cite[Th. 5]{Sim87} to conclude the strong convergence 
\begin{align*}
 A(\varphi_\e) \to A(\varphi) \; \mbox{ in } \; L^2(0,T;H^1(\Omega)) \,.
\end{align*}
From $\nabla A(\varphi_\e) = \sqrt{a(\varphi_\e)} \nabla \varphi_\e$ we get then
in particular the strong convergence 
\begin{align*}
 \nabla \varphi_\e \to \nabla \varphi \; \mbox{ in } \; L^2(0,T;L^2(\Omega)) \,.
\end{align*}
In addition we want to use an argument of Abels, Depner and Garcke from \cite[Sec. 5.1]{ADG12}
which shows that due to the a priori estimate in Lemma \ref{lem:energyestimate} and the structure
of equation \eqref{approxweakline1} we can deduce the strong convergence $\bv_\e \to \bv$ in $L^2(0,T;L^2(\Omega)^d)$. 
In few words we show with the help of some interpolation inequalities the bound of 
$\partial_t(P_\sigma(\rho_\e \bv_\e))$ in the space $L^{\frac{8}{7}}( W^1_\infty(\Omega)')$ and 
together with the bound of $P_\sigma(\rho_\e \bv_\e)$ in $L^2(0,T;H^1(\Omega)^d)$ 
this is enough to conclude with the Lemma of Aubin-Lions the strong convergence
\begin{align*}
 P_\sigma(\rho_\e \bv_\e) \to P_\sigma(\rho \bv) \; \mbox{ in } \; L^2(0,T;L^2(\Omega)^d) \,.
\end{align*}
From this we can derive $\bv_\e \to \bv$ in $L^2(0,T;L^2(\Omega)^d)$. 
For the details we refer to \cite[Sec. 5.1 and Appendix]{ADG12}.

With the last convergences and the weak convergence $\bJ_\e \rightharpoonup \bJ$ in $L^2(Q_T)$ we 
can pass to the limit $\e \to 0$ in line \eqref{approxweakline1} to achieve 
\eqref{weakline1}.

The convergence in line \eqref{approxweakline2} follows from the above weak limits of $\varphi_\e$ and 
$\bJ_\e$ in $L^2(Q_T)$ and the strong ones of $\bv_\e$ and $\nabla \varphi_\e$ in $L^2(Q_T)$.

Finally, the convergence in line \eqref{approxweakline3} can be seen as follows:
The left side converges due to the weak convergence of $\bJ_\e$ and for the right side we calculate
\begin{align} 
  \int_{Q_T} &\left(\Psi_\e'(\varphi_\e) - \sqrt{a(\varphi_\e)} \Delta A(\varphi_\e) \right) \di(m_\e(\varphi_\e) \boldsymbol{\eta}) \, dx \, dt \nonumber \\
  &= \int_{Q_T} \Psi'(\varphi_\e) \di (m_\e(\varphi_\e) \boldsymbol{\eta}) \, dx \, dt 
    + \e \int_{Q_T} \Psi_{\mbox{\footnotesize ln}}'(\varphi_\e) \di (m_\e(\varphi_\e) \boldsymbol{\eta}) \, dx \, dt \nonumber \\
  & \quad - \int_{Q_T} \sqrt{a(\varphi_\e)} \Delta A(\varphi_\e) \di (m_\e(\varphi_\e) \boldsymbol{\eta}) \, dx \, dt \,. \label{thirdconverg}
\end{align}
The first and the third term can be treated similarly as in Elliott and Garcke \cite{EG96}. For the convenience of the 
reader we give the details.

First we observe the fact that $m_\e \to m$ uniformly since for all $s \in \R$ it holds:
\begin{align*}
 |m_\e(s) - m(s)| &\leq m(1-\e) \to 0 \quad \mbox{for } \e \to 0 \,.
\end{align*}
Hence we conclude with the pointwise convergence $\varphi_\e \to \varphi$ a.e. in $Q_T$ that
\begin{align*}
 m_\e(\varphi_\e) \to m(\varphi) \quad \mbox{a.e. in } \; Q_T \,.
\end{align*}
In addition with the convergences $\Psi'(\varphi_\e) \to \Psi'(\varphi)$, $a(\varphi_\e) \to a(\varphi)$
a.e. in $Q_T$ and with the weak convergence $\Delta A(\varphi_\e) \to \Delta A(\varphi)$ in $L^2(Q_T)$ we are led to 
\begin{align*}
 \int_{Q_T} \Psi'(\varphi_\e) m_\e(\varphi_\e) \di \boldsymbol{\eta} \, dx\,dt 
   &\longrightarrow \hspace*{-2pt}
  \int_{Q_T} \Psi'(\varphi) m(\varphi) \di \boldsymbol{\eta} \, dx\,dt \; \mbox{ and } \\
 \int_{Q_T} \hspace*{-3pt} \sqrt{a(\varphi_\e)} \Delta A(\varphi_\e) m_\e(\varphi_\e) \di \boldsymbol{\eta} \, dx \hspace*{1pt} dt 
   &\longrightarrow \hspace*{-2pt}
  \int_{Q_T} \hspace*{-3pt} \sqrt{a(\varphi)} \Delta A(\varphi) m(\varphi) \di \boldsymbol{\eta} \, dx \hspace*{1pt} dt .
\end{align*}
The next step is to show that $m_\e'(\varphi_\e) \nabla \varphi_\e \to m'(\varphi) \nabla \varphi$ in $L^2(Q_T)$. 
To this end we split the integral in the following way:
\begin{align*}
 \int_{Q_T} & |m_\e'(\varphi_\e) \nabla \varphi_\e - m'(\varphi) \nabla \varphi|^2 \, dx\,dt \\
  &= \int_{Q_T \cap \{|\varphi| < 1\}} |m_\e'(\varphi_\e) \nabla \varphi_\e - m'(\varphi) \nabla \varphi|^2 \, dx\,dt \\
    & \quad + \int_{Q_T \cap \{|\varphi| = 1\}} |m_\e'(\varphi_\e) \nabla \varphi_\e - m'(\varphi) \nabla \varphi|^2 \, dx\,dt \,.
\end{align*}
Since $\nabla \varphi = 0$ a.e. on the set $\{|\varphi| = 1\}$, see for example Gilbarg and Trudinger \cite[Lem.~7.7]{GT01},
we obtain
\begin{align*}
 &\int_{Q_T \cap \{|\varphi| = 1\}} |m_\e'(\varphi_\e) \nabla \varphi_\e - m'(\varphi) \nabla \varphi|^2 \, dx\,dt \\
  & \;\; = \int_{Q_T \cap \{|\varphi| = 1\}} |m_\e'(\varphi_\e) \nabla \varphi_\e|^2 \, dx\,dt \\
 & \;\; \leq C \int_{Q_T \cap \{|\varphi| = 1\}} |\nabla \varphi_\e|^2 \, dx\,dt \longrightarrow 
         C \int_{Q_T \cap \{|\varphi| = 1\}} |\nabla \varphi|^2 \, dx\,dt = 0 \,.
\end{align*}
Although $m_\e'$ is not continuous, we can conclude on the set $\{|\varphi_\e| < 1\}$ the convergence
$m_\e'(\varphi_\e) \to m'(\varphi)$ a.e. in $Q_T$. Indeed, for a point $(x,t) \in Q_T$ with $|\varphi(x,t)| < 1$
and $\varphi_\e(x,t) \to \varphi(x,t)$, it holds that $|\varphi_\e(x,t)| < 1 - \delta$ for some $\delta > 0$ and $\e$ small enough and in 
that region $m_\e'$ and $m'$ are continuous. Hence we have
\begin{align}
 m_\e'(\varphi_\e) \nabla \varphi_\e \longrightarrow m'(\varphi) \nabla \varphi \quad \mbox{ a.e. in } \; Q_T \label{convpointwiseme}
\end{align}
and the generalized Lebesgue convergence theorem now gives
\begin{align*}
 \int_{Q_T \cap \{|\varphi| < 1\}} |m_\e'(\varphi_\e) \nabla \varphi_\e - m'(\varphi) \nabla \varphi|^2 \, dx \,dt
   \longrightarrow 0 \,,
\end{align*}
which proves finally $m_\e'(\varphi_\e) \nabla \varphi_\e \to m'(\varphi) \nabla \varphi$ in $L^2(Q_T)$.
Similarly as above, together with the convergences $\Psi'(\varphi_\e) \to \Psi'(\varphi)$, $a(\varphi_\e) \to a(\varphi)$
a.e. in $Q_T$ and with the weak convergence $\Delta A(\varphi_\e) \to \Delta A(\varphi)$ in $L^2(Q_T)$ we are led to 
\begin{align*}
 \int_{Q_T} & \Psi'(\varphi_\e) m_\e'(\varphi_\e) \nabla \varphi_\e \cdot \boldsymbol{\eta} \, dx\,dt \\ 
   &\longrightarrow 
  \int_{Q_T} \Psi'(\varphi) m'(\varphi) \nabla \varphi \cdot \boldsymbol{\eta} \, dx\,dt \; \mbox{ and } \\
 \int_{Q_T} & \sqrt{a(\varphi_\e)} \Delta A(\varphi_\e) m_\e'(\varphi_\e) \nabla \varphi_\e \cdot \boldsymbol{\eta} \, dx \, dt \\
  &\longrightarrow
   \int_{Q_T} \sqrt{a(\varphi)} \Delta A(\varphi) m'(\varphi) \nabla \varphi \cdot \boldsymbol{\eta} \, dx \, dt .
\end{align*}
Now we are left to show that the second term of the right side in \eqref{thirdconverg} converges to zero. 
To this end, we split it in the following way:
\begin{align*}
 & \e \int_{\Omega_T} \Psi_{\mbox{\footnotesize ln}}'(\varphi_\e) \di (m_\e(\varphi_\e) \boldsymbol{\eta}) \, dx \, dt  \\
   &= \e \int_{\{|\varphi_\e| \leq 1 - \e\}} \Psi_{\mbox{\footnotesize ln}}'(\varphi_\e) \di (m_\e(\varphi_\e) \boldsymbol{\eta}) \, dx \, dt \\
   & \quad + \e \int_{\{|\varphi_\e| > 1 - \e\}} \Psi_{\mbox{\footnotesize ln}}'(\varphi_\e) \di (m_\e(\varphi_\e) \boldsymbol{\eta}) \, dx \, dt  \\
   &=: (I)_\e + (II)_\e \,. 
\end{align*}
On the set $\{|\varphi_\e| \leq 1 - \e\}$ we use that $\Psi_{\mbox{\footnotesize ln}}'(\varphi_\e) = \ln(1+\varphi_\e) - \ln(1-\varphi_\e) + 2$ and
therefore $\left| \Psi_{\mbox{\footnotesize ln}}'(\varphi_\e) \right| \leq |\ln \e | + C$ to deduce that
\begin{align*}
 |(I)_\e| & \leq \e (|\ln \e| + C) \int_{Q_T} |\di (m_\e(\varphi_\e) \boldsymbol{\eta}) | \, dx \, dt \longrightarrow 0 \,.
\end{align*}
On the set $\{|\varphi_\e| > 1 - \e\}$, we use that $m_\e(\varphi_\e) = \e (2-\e)$ to deduce
\begin{align*}
 (II)_\e &= \e^2 (2-\e) \int_{\{|\varphi_\e| > 1 - \e\}} \Psi_{\mbox{\footnotesize ln}}'(\varphi_\e) \di \boldsymbol{\eta} \, dx \, dt \\
  & \leq C \e^2 \| \Psi_{\mbox{\footnotesize ln}}'(\varphi_\e) \|_{L^2(Q_T)} \\
  & = C \sqrt{\e} \left( \e^{\frac{3}{2}} \| \Psi_{\mbox{\footnotesize ln}}'(\varphi_\e) \|_{L^2(Q_T)} \right) 
  \longrightarrow 0 \,,
\end{align*}
since the last term in brackets is bounded by the energy estimate form Lemma \ref{lem:energyestimate}.

For the relation of $\widehat{\bJ}$ and $\bJ$ we note that due to $\widehat{\bJ}_\e \rightharpoonup \widehat{\bJ}$,
$\bJ_\e \rightharpoonup \bJ$ in $L^2(Q_T)$, $\bJ_\e = \sqrt{m_\e(\varphi_\e)} \, \widehat{\bJ}_\e$ and 
$\sqrt{m_\e(\varphi_\e)} \to \sqrt{m(\varphi)}$ a.e. in $Q_T$ from \eqref{convpointwiseme} we can conclude
\begin{align*}
 \bJ &= \sqrt{m(\varphi)} \, \widehat{\bJ} \,.
\end{align*}
From the weak convergence $\widehat{\bJ}_\e \rightharpoonup \widehat{\bJ}$ in $L^2(Q_T)$ we can conclude that
\begin{align*}
 \int_{Q_(s,t)} |\widehat{\bJ}|^2 \, dx \, d\tau & \leq 
  \liminf_{\e \to 0} \int_{Q_(s,t)} m_\e(\varphi_\e) |\nabla \mu_\e|^2 \, dx \, d\tau 
\end{align*}
for all $0 \leq s \leq t \leq T$
and this is enough to proceed as in Abels, Depner and Garcke \cite{ADG12} to show the energy estimate. 

Finally we just remark that the continuity properties and the initial conditions can be derived
with the same arguments as in \cite[Sec. 5.2, 5.3]{ADG12}, so that altogether we proved 
Theorem~\ref{theo:existenceweaksolution}.

%%%%%%%%%%%%%%%%%%%%%%%%%%%%%%%%%%%%%%%%%%%%%%%%%%%%%%%%%%%%%%%%%%%%%%%%%%%%%%%%%%%%%%%%%%%%%%%%%%%%%%%%%%
%%%%%%%%%%%%%%%%%%%%%%%%%%%%%%%%%%%%%%%%%%%%%%%%%%%%%%%%%%%%%%%%%%%%%%%%%%%%%%%%%%%%%%%%%%%%%%%%%%%%%%%%%%
%%%%%%%%%%%%%%%%%%%%%%%%%%%%%%%%%%%%%%%%%%%%%%%%%%%%%%%%%%%%%%%%%%%%%%%%%%%%%%%%%%%%%%%%%%%%%%%%%%%%%%%%%%
\section*{Acknowledgement} 
This work was supported by the SPP 1506 "Transport Processes
at Fluidic Interfaces" of the German Science Foundation (DFG) through
the grant GA 695/6-1. The support is gratefully acknowledged.

%%%%%%%%%%%%%%%%%%%%%%%%%%%%%%%%%%%%%%%%%%%%%%%%%%%%%%%%%%%%%%%%%%%%%%%%%%%%%%%%%%%%%%%%%%%%%%%%%%%%%%%%%%
%%%%%%%%%%%%%%%%%%%%%%%%%%%%%%%%%%%%%%%%%%%%%%%%%%%%%%%%%%%%%%%%%%%%%%%%%%%%%%%%%%%%%%%%%%%%%%%%%%%%%%%%%%
%%%%%%%%%%%%%%%%%%%%%%%%%%%%%%%%%%%%%%%%%%%%%%%%%%%%%%%%%%%%%%%%%%%%%%%%%%%%%%%%%%%%%%%%%%%%%%%%%%%%%%%%%%

\end{document}